\documentclass[11pt, a4paper]{amsart} 

\usepackage{amsmath, amsthm, amssymb, fullpage}

\newcommand{\Vv}{\boldsymbol{v}}
\newcommand{\bN}{\mathbb{N}}
\newcommand{\bR}{\mathbb{R}}
\newcommand{\bP}{\mathbb{P}}
\newcommand{\bC}{\mathbb{C}}
\newcommand{\bD}{\mathbb{D}}
\newcommand{\bZ}{\mathbb{Z}}
\newcommand{\bQ}{\mathbb{Q}}
\newcommand{\rd}{\mathrm{d}}
\newcommand{\loc}{\mathrm{loc}}

\theoremstyle{plain}

\newtheorem{mainth}{Theorem}
\theoremstyle{definition}

\newtheorem*{notation}{Notation}
\newtheorem*{acknowledgement}{Acknowledgement}
\theoremstyle{remark}

\numberwithin{equation}{section}

\begin{document} 
\title[]{Equidistribution of the zeros of higher order derivatives in polynomial dynamics}

\author{Y\^usuke Okuyama}
\address{Division of Mathematics, Kyoto Institute of Technology, Sakyo-ku, Kyoto 606-8585 JAPAN}
\email{okuyama@kit.ac.jp}

\date{\today}

\subjclass[2010]{Primary 37F10: Secondary 30D05}

\keywords{equidistribution, higher order derivatives, iterated polynomial,
Schr\"oder equation, Abel equation, complex dynamics}

\begin{abstract}
For every $m\in\bN$, we establish the convergence of the
averaged distributions of the zeros of
the $m$-th order derivatives $(f^n)^{(m)}$ 
of the iterated polynomials $f^n$ of a polynomial $f\in\bC[z]$ of degree $>1$ 
towards the harmonic measure of the filled-in Julia set of $f$ with pole at $\infty$
as $n\to+\infty$, when $f$ has no exceptional points in $\bC$.
This complements our former study on the zeros of 
$(f^n)^{(m)}-a$ for any value $a\in\bC\setminus\{0\}$.
The key in the proof is an approximation of the higher order derivatives of a solution of the Schr\"oder or Abel functional equations for a meromorphic 
function on $\bC$
with a locally uniform non-trivial error estimate.
\end{abstract}

\maketitle 

\section{Introduction}\label{sec:intro}
Let $f\in\bC[z]$ be a polynomial of degree $d>1$. The {\em filled-in Julia set} 
\begin{gather*}
K(f):=\Bigl\{z\in\bC:\limsup_{n\to+\infty}|f^n(z)|<+\infty\Bigr\} 
\end{gather*}
of $f$ is a non-polar compact subset in $\bC$, where $f^n$ denotes the $n$-th iteration of $f$ for each $n\in\bN=\{1,2,3,\ldots\}$. 
Let $g_f$ be the Green function of $K(f)$ with pole at $\infty$
(for the details on potential theory, see e.g.\ \cite[\S4.4]{Ransford95}), 
where the projective line $\bP^1=\bP^1(\bC)$ is regarded as the Riemann sphere $\bC\cup\{\infty\}$ fixing an affine chart of $\bP^1$.
The function $g_f$ extends to $\bC$ continuously by setting $\equiv 0$ on $K(f)$ (so that 
$\bC\setminus K(f)=\{g_f>0\}$).
The difference $g_f-(\log\max\{1,|f^n|\})/d^n$ on $\bC$
is harmonic and bounded near $\infty$ so it admits a harmonic
extension near $\infty$ for each $n\in\bN$, and then
is $=O(d^{-n})$ as $n\to+\infty$ on $\bP^1$ uniformly.
The {\em harmonic measure} of $K(f)$ with pole at $\infty$
is the probability measure
\begin{gather*}
\mu_f:=\Delta g_f\quad\text{on }\bC,
\end{gather*}
which has no atoms on $\bP^1$ and is supported by $\partial(K(f))$.
The {\em exceptional set} of the polynomial $f$ is defined as
\begin{gather*}
 E(f):=\Bigl\{a\in\bP^1:\#\bigcup_{n\in\bN\cup\{0\}}f^{-n}(a)<+\infty\Bigr\},
\end{gather*}
which consists of $\infty$ (indeed $f^{-1}(\infty)=\{\infty\}$)
and at most one point in $\bC$, and
a point $b\in\bC$ belongs to $E(f)$ if and only if $f(z)=A(z-b)^d+b$ 
for some $A\in\bC\setminus\{0\}$ (and then $f^{-1}(b)=\{b\}$).
The measure $\mu_f$ is characterized as the unique probability measure $\nu$ on $\bP^1$
such that 
\begin{gather*}
 f^*\nu=d\cdot\nu\quad\text{on }\bP^1\quad\text{and that}\quad \nu(E(f))=0, 
\end{gather*}
where $f^*$ is the pullback operator 
on the space of complex Radon measure on $\bP^1$ induced by $f$
(see Notation below for the details). 
Brolin \cite{brolin} studied the value distribution
of the iteration sequence $(f^n)$ of $f$,
and established that
{\em for every $a\in\bP^1\setminus E(f)$,}
\begin{gather}
 \lim_{n\to+\infty}\frac{(f^n)^*\delta_a}{d^n}=\mu_f\quad\text{\em weakly on }\bP^1,\label{eq:Brolin}
\end{gather}
where $\delta_z$ is the Dirac measure on $\bP^1$ 
at a point $z\in\bP^1$.
This equidistribution result \eqref{eq:Brolin} is foundational in the value distribution theory 
in complex dynamics; 
see e.g.\ \cite{BD99, DS08, FLM83, Lyubich83, RS97}.

In \cite{OVderivative},
for every $m\in\bN$, a similar equidistribution 
statement towards $\mu_f$ replacing $f^n$ with the $m$-th order derivative 
\begin{gather*}
 (f^n)^{(m)}=\frac{\rd^m}{\rd z^m}f^n(z) \in\bC[z]\quad\text{of degree }d^n-m
\end{gather*}
of $f^n$, $n\in\bN$, 
has been established for any value $a\in\bP^1\setminus\{0,\infty\}$
(and the $m=1$ case was first treated in \cite{GV16derivative, OkuyamaDerivatives}). In this value distribution result for the sequence
$((f^n)^{(m)})_n$ of the $m$-th order derivatives of iterated polynomials $f^n$,
the value $a=\infty$ is always exceptional 
since for any $n\in\bN$, 
$((f^n)^{(m)})^*\delta_\infty/(d^n-m)=\delta_\infty\neq\mu_f$ on $\bP^1$, and the 
(maybe most important) value
$a=0$ is also exceptional unless $E(f)=\{\infty\}$
(from the above description of the case $E(f)\neq\{\infty\}$).

Our principal result asserts that conversely, if $E(f)=\{\infty\}$, then
for any $m\in\bN$, the value $a=0$ is not exceptional
in our study of the value distributions of the higher order derivatives of iterated
polynomials.

\begin{mainth}\label{th:anyorder}
Let $f\in\bC[z]$ be a polynomial of degree $d>1$, and $m\in\bN$.
If $E(f)=\{\infty\}$, then
\begin{gather}
 \lim_{n\to\infty}\frac{\bigl((f^n)^{(m)}\bigr)^*\delta_0}{d^n-m}=\mu_f\quad\text{weakly on }\bP^1.\label{eq:zeros}
\end{gather}
\end{mainth}

For a meromorphic function $f$ on $\bC$, 
the Fatou set $F(f)$ of $f$ is defined by the region of normality
for the restriction 
$f:\bC\setminus\overline{\bigcup_{n\in\bN}f^{-n}(\infty)}
\to\bC\setminus\overline{\bigcup_{n\in\bN}f^{-n}(\infty)}$,
so that $f$ has no pole in $F(f)$.
To exclude easy cases and avoid superficial complications,
we assume that when $f$ extends to a rational function on $\bP^1$,
then $\deg >1$ and either $f$ is a polynomial or the above $F(f)$
coincides with the usual Fatou set in $\bP^1$ of this extended $f$.

The following 
approximation of the higher order derivatives $\Phi^{(t)}(z)=\frac{\rd^t}{\rd z^t}\Phi(z)$
of a solution $\Phi(z)$ of the Schr\"oder or Abel 
functional equation for a meromorphic function $f$ on $\bC$ having a 
locally uniform non-trivial error estimate
is also one of our principal results. This plays a key role in 
the proof of Theorem \ref{th:anyorder} and might be of independent interest.

\begin{mainth}\label{th:convergence}
 Let $f$ be a meromorphic function on $\bC$, and
 $a\in\bC$ be a fixed point of $f^p$ for some $p\in\bN$
 having the multiplier
 $\lambda:=(f^p)'(a)\in(\bD\setminus\{0\})\cup\{1\}$. 
 Then for every $t\in\bN\setminus\{1\}$,
\begin{gather}
 \frac{(f^n)^{(t)}}{(f^n)'}
 =\frac{\Phi^{(t)}}{\Phi'}
 +\begin{cases}
   \displaystyle O\bigl(|\lambda|^{\frac{n}{p}}\bigr) &\text{if }\lambda\in\bD\setminus\{0\},\\
   \displaystyle O\Bigl(\frac{1}{n}\Bigr) &\text{if }\lambda=1
  \end{cases}\quad\text{as }n\to+\infty
\label{eq:locunif}
\end{gather}
locally uniformly on $\Omega\setminus\{\Phi'=0\}$, where $\Omega\subset\bC$ is the 
either attracting or parabolic basin of $f$ associated to
$($the orbits under $f$ of the periodic point$)$ $a$, and where
the holomorphic surjection
$\Phi:\Omega\to\bC$ is a solution of the Schr\"oder or Abel equation
for $f$ associated to $($the orbits under $f$ of the periodic point$)$ $a$, respectively.
\end{mainth}

For the details on the Schr\"oder or Abel equations, see Subsection \ref{sec:preliminary}. We note that a statement similar to that in Theorem \ref{th:convergence} is also true for a polynomial-like mapping $f:U\to V$ for simply connected domains $U\Subset V$ in $\bC$
and that the general parabolic case of
$\lambda\in\{e^{2i\pi\theta}:\theta\in\bQ\}$
can be treated in a way similar to the above multiple case of $\lambda=1$.

\begin{notation}
A rational function $h\in\bC(z)$ of degree $>0$
induces a pullback operator $h^*$ on the space of complex Radon measures on $\bP^1$,
so that for every $a\in\bP^1$, 
\begin{gather*}
 h^*\delta_a=\sum_{w\in h^{-1}(a)}(\deg_w h)\delta_w\quad\text{on }\bP^1, 
\end{gather*}
where $\deg_z h$ is
the local degree of $h$ at each point $z\in\bP^1$ and $\delta_z$ denotes
the Dirac measure on $\bP^1$ at each point $z\in\bP^1$.
The open unit disk in $\bC$ is denoted as $\bD=\{z\in\bC:|z|<1\}$.
\end{notation}

\subsection*{Organization of the article}
In Section \ref{sec:preliminary}, we do
some preparatory computations on the higher order derivatives of iterations
of a meromorphic function on $\bC$ restricted to 
an either attracting or parabolic basin of it
using a solution of the Schr\"oder or Abel functional equation,
and show Theorem \ref{th:convergence}.
In Section \ref{sec:overC}, 
using Theorem \ref{th:convergence},
we show Theorem \ref{th:potential} below, which is equivalent
to Theorem \ref{th:anyorder}.

\section{Proof of Theorem \ref{th:convergence}}\label{sec:preliminary}
Let $f$ be a meromorphic function on $\bC$,
and recall the definition of the Fatou set $F(f)$ of $f$ 
in Section \ref{sec:intro}.
For the details on complex dynamics, see e.g.\ \cite{Milnor3rd}. 

\subsection{Schr\"oder and Abel functional equations}\label{sec:general}
If $f$ 
has an either (super)attracting or multiple (so in particular parabolic) 
periodic point 
$a\in\bC\setminus\bigcup_{n\in\bN}f^{-n}(\infty)$ having the period
\begin{gather*}
 p=p_{f,a}\in\bN,
\end{gather*}
then the multiplier
\begin{gather*}
 \lambda=\lambda_{f,a}:=(f^p)'(a)
\end{gather*}
of $f^p$ at the fixed point $a$ is either in $\bD$ or equals $1$, 
respectively, and
the (super)attracting or parabolic basin 
\begin{gather*}
 \Omega=\Omega_{f,a}:=\Bigl\{z\in F(f):\lim_{n\to+\infty}f^{pn+j}(z)=a\quad\text{for some }j\in\bN\cup\{0\}\Bigr\} 
\end{gather*}
of $f$ associated to (the orbits under $f$ of the periodic point) $a$ is completely invariant under $f$.

If in addition $\lambda\neq 0$, 
that is, the above fixed point $a$ of $f^p$ is not superattracting,
then by the Koenigs-Schr\"oder-Poincar\'e theorem or
the Leau-Fatou flower theorem, respectively, there is a holomorphic surjection 
\begin{gather*}
 \Phi=\Phi_{f,a}:\Omega\to\bC
\end{gather*}
satisfying the Schr\"oder or Abel equation
\begin{gather}
\begin{cases}
 \Phi\circ f=\lambda^{1/p}\cdot\Phi\quad\text{on }\Omega
\quad\text{and}
\quad\Phi'(f^j(a))\neq 0\text{ for every }
j\in\{0,\ldots,p-1\} & \text{if }\lambda\in\bD\setminus\{0\},\\
 \Phi\circ f=\Phi+\frac{1}{p}\quad\text{on }\Omega 
\quad\text{and }\Phi\text{ is injective on }P & \text{if }\lambda=1
\end{cases}\label{eq:functional}
\end{gather}
for $f$ associated to (the orbits under $f$ of the periodic point) $a$, where a $p$-th root $\lambda^{1/p}$ of $\lambda$ is fixed and
where $P$ is any attracting petal of $f^p$ associated to one of the at most finitely many
attracting directions towards the multiple fixed point $f^j(a)$ of $f^p$ for some $j\in\{0,\ldots,p-1\}$,
respectively; this 
$\Phi$ is unique up to multiplication in $\bC^*$ or addition in $\bC$, respectively.

\subsection{Preparatory computations} 
Suppose that $\lambda\in(\bD\setminus\{0\})\cup\{1\}$,
and fix a $p$-th root $\lambda^{1/p}\in\bigl(\bD\setminus\{0\}\bigr)\cup\{1\}$ of $\lambda$, 
assuming $\lambda^{1/p}=1$ when $\lambda=1$ by convention. 
By \eqref{eq:functional}, we have
\begin{gather}
(\Phi\circ f^n)'
=(\Phi'\circ f^n)\cdot(f^n)'=\lambda^{\frac{n}{p}}\cdot\Phi'\quad\text{on }\Omega\label{eq:chain}
\end{gather}
for every $n\in\bN$, using the second equality in which,
we have
\begin{gather}
\bigl\{\Phi'=0\bigr\}=\bigcup_{n\in\bN\cup\{0\}}f^{-n}\bigl(\{z\in\Omega:f'(z)=0\}\bigr).\label{eq:zerosbasin}
\end{gather}
Let us write $f^n=\Phi^{-1}\circ(\Phi\circ f^n)$, so that
$(f^n)'=((\Phi^{-1})'\circ\Phi\circ f^n)\cdot(\Phi\circ f^n)'$
for $n\gg 1$
locally uniformly on $\Omega$, and then
for every $t\in\bN$,
\begin{gather}
(f^n)^{(t)}
=\sum_{s=0}^{t-1}\begin{pmatrix}
		  t-1\\
		  s
		 \end{pmatrix}
\bigl((\Phi^{-1})'\circ\Phi\circ f^n\bigr)^{(s)}
\cdot\bigl(\Phi\circ f^n\bigr)^{(t-s)}\quad\text{for }n\gg 1\label{eq:binomial}
\end{gather}
locally uniformly on $\Omega$; here
once and for all, 
\begin{itemize}
 \item  when $\lambda\in\bD\setminus\{0\}$,
for each $j\in\{0,\ldots,p-1\}$,
we fix a topological open disk in $\Omega$ 
which contains the attracting
       fixed point $f^j(a)$ of $f^p$
       and is mapped 
 by $\Phi$ injectively onto an open round disk in $\bC$
 centered at $0(=\Phi(f^j(a)))$, or 
 \item when $\lambda=1$,
for each $j\in\{0,\ldots,p-1\}$ and
to each one say $\Vv$ of at most finitely many 
attracting directions of $f^p$ towards
the multiple fixed point $f^j(a)$,
we fix 
an attracting petal of $f^p$ which is associated 
to $\Vv$,
\end{itemize}
so that 
the function $\Phi^{-1}$ is the inverse of the restriction $\Phi|U$
of $\Phi$ to some $U$ among those 
$p$ topological open disks 
or at most finitely many attracting petals.

Let us define some polynomials
\begin{gather*}
 A_{s,u}\in\bZ[X_0,\ldots,X_s],\quad s\in\{0\}\cup\bN\text{ and }u\in\{0,\ldots,s\},
\end{gather*}
by $A_{0,0}:\equiv 1$ and,
inductively for $s\in\bN$, by the equality
\begin{multline*}
A_{s,u}(X_0,\ldots,X_s)
=A_{s-1,u-1}(X_0,\ldots,X_{s-1})\cdot X_0+\\
+\sum_{q=0}^{s-1}
\bigl((\partial_{X_q}(A_{s-1,u}))(X_0,\ldots,X_{s-1})\bigr)\cdot X_{q+1}
\quad\text{for }u\in\{0,\ldots,s\}, 
\end{multline*}
also setting the auxiliary $A_{s,-1}:\equiv 0$ for every $s\in\{0\}\cup\bN$
and $A_{s-1,s}:\equiv 0$ for every $s\in\bN$;
then for every $s\in\bN$ and every $u\in\{0,\ldots,s\}$,
\begin{gather}
 A_{s,u}(0,\ldots,0)=0,\label{eq:vanishconst} 
\end{gather}
and
we indeed have not only $A_{s,u}\in\bZ[X_0,\ldots,X_{s-1}]$
for every $s\in\bN$ and every $u\in\{0,\ldots,s\}$ but also
the stronger
\begin{gather}
 A_{s,0}\equiv 0\quad\text{for every }s\in\bN\label{eq:vanish}
\end{gather}
than \eqref{eq:vanishconst} when $u=0$, inductively.

Then for every $s\in\{0\}\cup\bN$, we inductively have
\begin{gather*}
\bigl((\Phi^{-1})'\circ\Phi\circ f^n\bigr)^{(s)}
=\sum_{u=0}^s
\bigl((\Phi^{-1})^{(u+1)}\circ\Phi\circ f^n\bigr)
\cdot A_{s,u}\bigl((\Phi\circ f^n)',\ldots,(\Phi\circ f^n)^{(s)}\bigr)
\quad\text{for }n\gg1
\end{gather*}
locally uniformly on $\Omega$(, reading 
$A_{0,0}(*,\ldots,*)\equiv 1$ here and below).
Hence for every $t\in\bN$, also by \eqref{eq:binomial} and
\eqref{eq:chain}, we have
\begin{multline*}
 (f^n)^{(t)}=\sum_{s=0}^{t-1}
 \begin{pmatrix}
  t-1\\
  s
 \end{pmatrix}\times\\
\times\biggl(\sum_{u=0}^s
\bigl((\Phi^{-1})^{(u+1)}\circ\Phi\circ f^n\bigr)
\cdot A_{s,u}\bigl(\lambda^{\frac{n}{p}}\cdot\Phi',\ldots,\lambda^{\frac{n}{p}}\cdot\Phi^{(s)}\bigr)\biggr)
\cdot\lambda^{\frac{n}{p}}\cdot\Phi^{(t-s)}\quad\text{for }n\gg 1
\end{multline*}
locally uniformly on $\Omega$, 
which together with \eqref{eq:vanish} implies
that, for every $t\in\bN\setminus\{1\}$,
\begin{multline}
\frac{(f^n)^{(t)}}{(f^n)'}=\frac{\Phi^{(t)}}{\Phi'}
+\sum_{s=1}^{t-1}\begin{pmatrix}
		  t-1\\
		  s
		 \end{pmatrix}\times\\
\times\biggl(\sum_{u=1}^s
\frac{(\Phi^{-1})^{(u+1)}\circ\Phi\circ f^n}{(\Phi^{-1})'\circ\Phi\circ f^n}
\cdot A_{s,u}\bigl(\lambda^{\frac{n}{p}}\cdot\Phi',\ldots,\lambda^{\frac{n}{p}}\cdot\Phi^{(s)}\bigr)
\biggr)\cdot\frac{\Phi^{(t-s)}}{\Phi'}
\quad\text{for }n\gg 1\label{eq:key}
\end{multline}
locally uniformly on $\Omega\setminus\{\Phi'=0\}$.

\begin{proof}[Proof of Theorem \ref{th:convergence}]
 If $\lambda\in\bD\setminus\{0\}$ (so $0<|\lambda|^{1/p}<1$), 
 then 
 we note that $\Phi\circ f^n=\lambda^{\frac{n}{p}}\cdot\Phi$ 
 on $\Omega$ by \eqref{eq:functional},
 that for every $s\in\bN$ and every $u\in\{1,\ldots,s\}$, 
\begin{gather*}
 A_{s,u}\bigl(\lambda^{\frac{n}{p}}\cdot\Phi',\ldots,\lambda^{\frac{n}{p}}\cdot\Phi^{(s)}\bigr)
  =O(|\lambda|^{\frac{n}{p}})\quad\text{as }n\to\infty 
\end{gather*} 
locally uniformly on $\Omega$ by \eqref{eq:vanishconst}, and that 
\begin{gather*}
 |(\Phi^{-1})'\circ\Phi\circ f^n|
>\frac{1}{2}\min\biggl\{\frac{1}{|\Phi'(f^j(a))|}:j\in\{0,\ldots,p-1\}\biggr\}>0\quad\text{for }n\gg 1
\end{gather*}
locally uniformly on $\Omega$ using the inverse function theorem.
Hence for every $t\in\bN\setminus\{1\}$, we have \eqref{eq:locunif} by \eqref{eq:key} in this case. 

On the other hand, if $\lambda=1$ (so $A_{s,u}(\lambda^{\frac{n}{p}}\cdot\Phi',\ldots,\lambda^{\frac{n}{p}}\cdot\Phi^{(s)})=A_{s,u}(\Phi',\ldots,\Phi^{(s)})$
under our convention that $\lambda^{1/p}=1$), then 
we note that $\Phi\circ f^n=\Phi+\frac{n}{p}$ on $\Omega$
by \eqref{eq:functional} 
and that the image $\Phi(P)$ under $\Phi$ of 
each $P$ among the attracting petals of $f^p$ that have been fixed once and for all contains a right half plane in $\bC$.
Hence for every $u\in\bN$, we have
 \begin{multline*}
 \limsup_{n\to+\infty}\Bigl|\frac{(\Phi^{-1})^{(u+1)}\circ\Phi\circ f^n}{(\Phi^{-1})'\circ\Phi\circ f^n}(\cdot)\Bigr|\cdot\frac{\bigl(R_n(\cdot)\bigr)^u}{(u+1)!}\\
 =\limsup_{n\to+\infty}\frac{\bigl|(\Phi^{-1})^{(u+1)}\bigl(\Phi(\cdot)+\frac{n}{p}\bigr)\cdot\bigl(R_n(\cdot)\bigr)^{u+1}/(u+1)!\bigr|}{\bigl|(\Phi^{-1})'\bigl(\Phi(\cdot)+\frac{n}{p}\bigr)\cdot R_n(\cdot)\bigr|}(\le u+1)<+\infty
\quad\text{for }n\gg 1
\end{multline*}
locally uniformly on $\Omega$
by choosing a function $R_n:\Omega\to\bR$
satisfying $R_n(\cdot)=\frac{n}{p}+O(1)$ as $n\to\infty$ locally uniformly 
on $\Omega$ so that for each $z_0\in\Omega$, 
the function
\begin{gather*}
\zeta\mapsto\Phi^{-1}\biggl(\Bigl(\Phi(z_0)+\frac{n}{p}\Bigr)+\bigl(R_n(z_0)\bigr)\cdot\zeta
\biggr)-\Phi^{-1}\Bigl(\Phi(z_0)+\frac{n}{p}\Bigr),\quad\text{which fixes }\zeta=0,
\end{gather*}
is defined and univalent on $\bD$ for $n\gg 1$, and then
applying e.g.\ the de Branges theorem to
the $(u+1)$-th coefficient of the Taylor expansion around $\zeta=0$ of
this univalent function on $\bD$; for the details on injective holomorphic (i.e.\ univalent) functions
 on $\bD$ including the de Branges theorem, see e.g.\ \cite{Haymanmult}.
 Hence for every $t\in\bN\setminus\{1\}$, we have \eqref{eq:locunif} by \eqref{eq:key}
 in this case.
\end{proof}

\section{Proof of Theorem \ref{th:anyorder}}\label{sec:overC}

Let $f\in\bC[z]$ be a polynomial of degree $d>1$. 
From now on, we denote by $m_2$ the real $2$-dimensional Lebesgue measure on $\bC$.
\begin{mainth}\label{th:potential}
If $E(f)=\{\infty\}$, then for every $m\in\bN$,
\begin{gather}
 \lim_{n\to+\infty}\frac{\log\bigl|(f^n)^{(m)}\bigr|}{d^n-m}=g_f\quad\text{in }L^1_{\loc}(\bC,m_2).\label{eq:potential}
\end{gather}
\end{mainth}
Here are a few remarks. For each $m\in\bN$, 
the convergence \eqref{eq:potential} implies
\eqref{eq:zeros}
by a continuity of the Laplacian $\Delta$ and the equalities
\begin{gather*}
\Delta\frac{\log\bigl|(f^n)^{(m)}\bigr|}{d^n-m}=\frac{\bigl((f^n)^{(m)}\bigr)^*\delta_0}{d^n-m}\quad\text{on }\bC\quad\text{for each }n\in\bN 
\end{gather*}
and $\Delta g_f=\mu_f$ on $\bC$.
Conversely, the convergence \eqref{eq:zeros} implies \eqref{eq:potential} since
the difference $g_f-(\log|(f^n)^{(m)}|)/(d^n-m)$ on $\bC$ 
not only extends harmonically near $\infty\in\bP^1$ 
for every $n\in\bN$
but also tends to $0$ as $n\to+\infty$ at the point $\infty$, as mentioned in Section \ref{sec:intro}.
As pointed out in \cite{OkuyamaDerivatives}, 
if $E(f)=\{\infty\}$, then the convergence
\eqref{eq:zeros} in the $m=1$ case
follows from the foundational ``$m=0$ case'' \eqref{eq:Brolin}; indeed,
no matter whether $E(f)=\{\infty\}$,
noting that $(f^n)'(z)=\prod_{j=0}^{n-1}f'(f^j(z))\in\bC[z]$ by the chain rule, 
for any integers $n\gg M\ge 0$, we compute
\begin{multline*}
 \frac{((f^n)')^*\delta_0}{d^n}-\mu_f
=\frac{1}{d^n}\sum_{c\in\bC:\,f'(c)=0}
\biggl(\sum_{j=0}^M(f^j)^*\delta_c
+\sum_{j=M+1}^{n-1}d^j\cdot\Bigl(\frac{(f^j)^*\delta_c}{d^j}-\mu_f\Bigr)
\biggr)
-d^{-n+M+1}\cdot\mu_f 
\end{multline*}
on $\bP^1$, the first sum in the right hand side in which
ranges over all $d-1$ critical points $c\in\bC$ of $f$
taking into account their multiplicities, and we note that
$(d-1)\cdot\sum_{j=M+1}^{n-1}d^j=d^n-d^{M+1}$.
If in addition $E(f)=\{\infty\}$, then we have
$(f^k)^*\delta_c/d^k-\mu_f\to 0$ weakly on $\bP^1$ as $k\to+\infty$
for every $c\in\bC$ (such that $f'(c)=0$) by \eqref{eq:Brolin}, so that
the convergence \eqref{eq:zeros} in the $m=1$ case 
holds making $n\to+\infty$ and in turn $M\to+\infty$ in the above displayed equality.

From now on, fix $m\in\bN$. 

\subsection{Observation from \cite{OVderivative}}
In \cite{OVderivative}, 
we have seen that the sequence $((\log|(f^n)^{(m)}|)/(d^n-m))_n$
of subharmonic functions on $\bC$ is locally uniformly upper bounded on $\bC$ 
and tends to $g_f$ as $n\to+\infty$ locally uniformly on 
$I_\infty(f)\setminus\bigcup_{n\in\bN\cup\{0\}}f^{-n}(C(f))$,
where 
\begin{gather*}
 I_\infty(f):=
\bP^1\setminus K(f)(=\{g_f>0\})
\end{gather*}
is the (immediate) superattracting basin
of $f$ associated to the superattracting fixed point $\infty$, and where 
$C(f):=\bigl\{z\in\bC:f'(z)=0\bigr\}\cup\{\infty\}$ 
is the set of all the critical (or branched) points of $f$
(as a holomorphic endomorphism of $\bP^1$) and 
$\#\bigl(C(f)\setminus\{\infty\}\bigr)\le d-1$.
Then by H\"ormander's {\em compactness principle} 
(\cite[Theorem 4.1.9(a)]{Hormander83}),
there are a subharmonic function $\phi$ on $\bC$ and
a strictly increasing sequence $(n_j)$ in $\bN$ tending to $+\infty$ as $j\to+\infty$ such that 
\begin{gather}
 \phi=\phi_m=\lim_{j\to+\infty}\frac{\log\bigl|(f^{n_j})^{(m)}\bigr|}{d^{n_{j}}-m}
\quad\text{in }L^1_{\operatorname{loc}}(\bC,m_2). 
\label{eq:subsequence}
\end{gather}

In \cite{OVderivative}, we then have seen that 
\begin{gather*}
 \phi\le g_f\text{ on }\bC\quad\text{and}\quad\phi\equiv g_f\text{ on }\overline{I_\infty(f)}\setminus\{\infty\}.
\end{gather*}
Moreover, unless the open subset $\{\phi<g_f\}$ in $\bC$,
which is in $K(f)$ and on which $g_f\equiv 0$, is empty, we have seen that any component 
of $\{\phi<0\}$ is a bounded Fatou component say $U$ 
of $f$(, i.e., a component of the interior of $K(f)$),
so that for every compact subset $C$ in $U$,
by H\"ormander's version of Hartogs's lemma (\cite[Theorem 4.1.9(b)]{Hormander83})
and the upper semicontinuity of $\phi$, we have
\begin{gather}
\limsup_{j\to+\infty}\frac{\sup_C\log\bigl|(f^{n_j})^{(m)}\bigr|}{d^{n_j}-m}
\le\sup_C\phi<0,\label{eq:hartogssmall}
\end{gather}
and we have also seen that there is a constant $a=a_U\in\bC$ such that
\begin{gather}
 \lim_{j\to+\infty}f^{n_j}=a\quad\text{locally uniformly on }U,\label{eq:constlimit}
\end{gather}
taking a subsequence of $(n_j)$ if necessary. 

\subsection{Further observation using complex dynamics}
Unless $\{\phi<g_f\}=\emptyset$, for the $U$ and $a=a_U$ in the last paragraph
in the above observation, 
by the no wandering Fatou component theorem 
and the classification theorem of cyclic Fatou components 
for a rational function on $\bP^1$ of degree $>1$ 
due to Sullivan and Fatou, respectively,
there is 
a cycle of Fatou components of $f$ having the period say $p\in\bN$ 
into which $U$ is mapped by $f^{n_j}$ for $j\gg 1$,
and moreover
$a$ is an either (super)attracting or multiple fixed point of $f^p$
and this cycle of Fatou components of $f$ is either
the immediate (super)attracting basin of $f$ or in the
immediate parabolic basin of $f$ associated to (the orbits under $f$ of) $a$.

\begin{proof}[Proof of Theorem \ref{th:potential}]
Let $f\in\bC[z]$ be a polynomial of degree $d>1$. 
Suppose that \eqref{eq:potential} does not hold for 
some $m\in\bN\setminus\{1\}$.
Under this assumption, there are a strictly increasing sequence 
$(n_j)$ in $\bN$ tending to $+\infty$
as $j\to+\infty$ and a subharmonic function 
$\phi=\phi_m$ on $\bC$ as in \eqref{eq:subsequence} for the $m$ 
such that the open subset $\{\phi<g_f\}$ in $K(f)$, on which $g_f\equiv 0$, is non-empty, and then
taking a subsequence of $(n_j)$ if necessary, 
there are $U,a=a_U$, and $p$
for this $\phi$ as in the above observations.
Recall the definition of the multiplier
$\lambda=\lambda_{f,a}:=(f^p)'(a)\in\bD\cup\{1\}$ 
and that of the (super)attracting or parabolic 
basin $\Omega=\Omega_{f,a}\subset\bC$,
of which $U$ is a component, and
the definition \eqref{eq:functional} of 
the holomorphic surjection $\Phi=\Phi_{f,a}:\Omega\to\bC$ 
associated to (the orbits under $f$ of) $a$ 
when $\lambda\in(\bD\setminus\{0\})\cup\{1\}$.

\subsection*{The case of $\lambda=0$} 
Suppose first that $\lambda=0$, that is, the fixed point $a$ of $f^p$ is superattracting. Then 
there is $z_0\in U\cap\bigcup_{n\in\bN\cup\{0\}}f^{-n}(a)$, so that
setting $N:=\min\{n\in\bN\cup\{0\}:f^n(z_0)=a\}$, we have
$\deg_{z_0}(f^{n_j})\ge(\deg_a(f^p))^{\frac{n_j-N}{p}-1}\to+\infty$ as $j\to+\infty$, so in particular
\begin{gather*}
 (f^{n_j})^{(m-1)}(z_0)=0\quad\text{for }j\gg 1.
\end{gather*}
Hence for $0<r\ll 1$, if $j\gg 1$, 
then using also $((f^n)^{(m-1)})'=(f^n)^{(m)}$, 
we have
\begin{gather*}
 \sup_{|z-z_0|\le r}\bigl|(f^{n_j})^{(m-1)}(z)\bigr|
=\sup_{|z-z_0|\le r}\bigl|(f^{n_{j}})^{(m-1)}(z)-(f^{n_{j}})^{(m-1)}(z_0)\bigr|
\le r\cdot\sup_{|z-z_0|\le r}\bigl|(f^{n_j})^{(m)}(z)\bigr|,
\end{gather*}
which with \eqref{eq:hartogssmall} yields
\begin{gather*}
\limsup_{j\to+\infty}\frac{\sup_{|z-z_0|\le r}\log\bigl|(f^{n_{j}})^{(m-1)}\bigr|}{d^{n_{j}}-(m-1)}
\le
\limsup_{j\to+\infty}\frac{\sup_{|z-z_0|\le r}\log\bigl|(f^{n_{j}})^{(m)}\bigr|}{d^{n_{j}}-m}
<0.
\end{gather*}

Now taking a subsequence of $(n_j)$ if necessary, for $m-1$,
the subharmonic function
\begin{gather*}
 \phi_{m-1}=\lim_{j\to+\infty}\frac{\log\bigl|(f^{n_j})^{(m-1)}\bigr|}{d^{n_{j}}-(m-1)}
\quad\text{in }L^1_{\operatorname{loc}}(\bC,m_2)
\end{gather*}
on $\bC$ as in \eqref{eq:subsequence} also
exists and is $<0\equiv g_f$ $m_2$-a.e.\ on $\overline{\bD(z_0,r)}\subset K(f)$,
so that \eqref{eq:potential} does not hold for $m-1$.
Inductively, \eqref{eq:potential} cannot hold for $m=1$, and we are done in this case of $\lambda=0$ by the remarks after 
Theorem \ref{th:potential}.

\subsection*{The case of $\lambda\neq 0$}
Suppose next that $\lambda\in(\bD\setminus\{0\})\cup\{1\}$. 
Then there is $z_0\in U\setminus\{\Phi'\cdot\Phi^{(m)}=0\}$
since the holomorphic surjection $\Phi:\Omega\to\bC$ 
does not extend to $\bC$ as a polynomial
by \eqref{eq:functional} and $\deg f>1$. 
Hence for $0<r\ll 1$, by 
Theorem \ref{th:convergence}
and \eqref{eq:hartogssmall}, we have
\begin{gather*}
 \limsup_{j\to+\infty}\frac{\sup_{|z-z_0|\le r}\log\bigl|(f^{n_j})'(z)\bigr|}{d^{n_j}-1}
=\limsup_{j\to+\infty}\frac{\sup_{|z-z_0|\le r}\log\bigl|(f^{n_j})^{(m)}(z)\bigr|}{d^{n_j}-m}<0,
\end{gather*}
and then \eqref{eq:potential} cannot hold for $m=1$ by an argument similar to
that in the second paragraph in the case of $\lambda=0$.
Now we are also done in this case of 
$\lambda\in(\bD\setminus\{0\})\cup\{1\}$
by the remarks after 
Theorem \ref{th:potential}.
\end{proof}

\begin{acknowledgement}
The author thanks the referee
for a very careful scrutiny and invaluable comments.
The author was partially supported by JSPS Grant-in-Aid 
for Scientific Research (C), 19K03541 and (B), 19H01798.
\end{acknowledgement}

\def\cprime{$'$}

\end{document}